\documentclass[12pt]{article}
\usepackage[utf8]{inputenc}

\usepackage{amsmath}
\usepackage{amsthm}
\usepackage{amssymb}
\usepackage[title]{appendix}

\usepackage{hyperref}

\usepackage{cite}
\usepackage{color}

\usepackage{enumitem}
\setlist[enumerate]{label={\textnormal{(\roman*)}}}

\newtheorem{theorem}{Theorem}[section]
\newtheorem{lemma}[theorem]{Lemma}
\newtheorem{corollary}[theorem]{Corollary}
\newtheorem{proposition}[theorem]{Proposition}

\theoremstyle{definition}
\newtheorem{definition}[theorem]{Definition}
\newtheorem{remark}[theorem]{Remark}
\newtheorem{problem}[theorem]{Problem}
\newtheorem{conjecture}[theorem]{Conjecture}

\usepackage{tikz}
\tikzstyle{vertex}=[circle, draw, inner sep=0pt, minimum size=4pt,fill=black]

\usetikzlibrary{arrows}
\usetikzlibrary{decorations.markings}
\tikzset{->-/.style={decoration={
  markings,
  mark=at position #1 with {\arrow{>}}},postaction={decorate}}}
\renewcommand{\tt}[1]{\textnormal{\texttt{{#1}}}}

\newcommand{\Rext}{\mathbb{R}_{\text{ext}}}
\DeclareMathOperator{\RT}{RT}
\DeclareMathOperator{\ERT}{ERT}

\title{Extremal overlap-free and extremal $\beta$-free binary words}
\author{Lucas Mol and Narad Rampersad\thanks{The work of Narad Rampersad is supported by the Natural Sciences and Engineering Research Council of Canada (NSERC), [funding reference number 2019-04111].}\\
\small{Department of Mathematics and Statistics, The University of Winnipeg}\\
\small{\href{mailto:l.mol@uwinnipeg.ca}{l.mol@uwinnipeg.ca}, \href{mailto:n.rampersad@uwinnipeg.ca}{n.rampersad@uwinnipeg.ca}}\\ \\
Jeffrey Shallit\thanks{The work of Jeffrey Shallit is supported by the Natural Sciences and Engineering Research Council of Canada (NSERC), [funding reference number 2018-04118].}\\
\small{School of Computer Science, University of Waterloo}\\ \small{\href{mailto:shallit@uwaterloo.ca}{shallit@uwaterloo.ca}}}
\date{}

\begin{document}

\maketitle

\begin{abstract}
\noindent
An overlap-free (or $\beta$-free) word $w$ over a fixed alphabet $\Sigma$ is \emph{extremal} if every word obtained from $w$ by inserting a single letter from $\Sigma$ at any position contains an overlap (or a factor of exponent at least $\beta$, respectively).  We find all lengths which admit an extremal overlap-free binary word.  For every extended real number $\beta$ such that $2^+\leq\beta\leq 8/3$, we show that there are arbitrarily long extremal $\beta$-free binary words.



\end{abstract}

\section{Introduction}

Throughout, we use standard definitions and notations from combinatorics on words (see~\cite{LothaireAlgebraic}).  For every integer $n\geq 2$, we let $\Sigma_n$ denote the alphabet $\{\tt{0},\tt{1},\ldots,\tt{n-1}\}$.  The word $u$ is a \emph{factor} of the word $w$ if we can write $w=xuy$ for some (possibly empty) words $x,y$.  A \emph{square} is a word of the form $xx$, where $x$ is nonempty.  An \emph{overlap} is a word of the form $axaxa$, where $a$ is a letter and $x$ is a (possibly empty) word.  A word is \emph{square-free} if it contains no square as a factor, and \emph{overlap-free} if it contains no overlap as a factor.
Early in the twentieth century, Norwegian mathematician Axel Thue~\cite{Thue1906,Thue1912} demonstrated that one can construct arbitrarily long square-free words over a ternary alphabet, and arbitrarily long overlap-free words over a binary alphabet.  For an English translation of Thue's work, see~\cite{Berstel1995}. Thue's work is recognized as the beginning of the field of combinatorics on words~\cite{BerstelPerrin2007}.

Let $w$ be a word over a fixed alphabet $\Sigma$.  An \emph{extension} of $w$ is a word of the form $w'aw''$, where $a\in \Sigma$, and $w'w''=w$ for some possibly empty words $w',w''\in \Sigma^*$.  For example, over the English alphabet, the English word \tt{pans} has extensions including the English words \tt{spans}, \tt{plans}, \tt{pawns}, \tt{pants}, and \tt{pansy}.  The word $w$ is \emph{extremal square-free} if $w$ is square-free, and every extension of $w$ contains a square.   For example, the word
\[
\tt{abcabacbcabcbabcabacbcabc}
\]
of length $25$ is an extremal square-free word of minimum length over the alphabet $\{\tt{a},\tt{b},\tt{c}\}$.  The concept of extremal square-free word was recently introduced by Grytczuk et al.~\cite{Grytczuk2019}, who demonstrated that there are arbitrarily long extremal square-free words over a ternary alphabet.  Two of the present authors~\cite{MolRampersad2020} adapted their ideas to find all lengths admitting extremal square-free ternary words.  

In this paper, we consider some variations of extremal square-free words, with a focus on the binary alphabet $\Sigma_2=\{\tt{0},\tt{1}\}$.  We begin by considering extremal overlap-free words, as suggested by Grytczuk et al.~\cite{Grytczuk2019}.  For a word $w$ over a fixed alphabet $\Sigma$, we say that $w$ is \emph{extremal overlap-free} if $w$ is overlap-free, and every extension of $w$ contains an overlap.  For example, the word $\tt{0010011011}$ of length $10$ is an extremal overlap-free word of minimum length over $\Sigma_2$.

While there is an extremal square-free ternary word of every sufficiently large length, the same cannot be said for extremal overlap-free binary words.  Our first main result is the following characterization of the lengths of extremal overlap-free binary words.

\begin{theorem}\label{main}
Let $n$ be a nonnegative number.  Then there is an extremal overlap-free word of length $n$ over the alphabet $\Sigma_2$ if and only if $n$ is in the set
$$\mathcal{N}:=\{10,12\}\cup\{2k \colon\ k\geq 10\} \cup \left\{2^k+1 \colon\ k \geq 5\right\} \cup \left\{3\cdot 2^k + 1 \colon\ k \geq 3\right\}.$$
\end{theorem}

After proving Theorem~\ref{main}, we consider a more general problem, which we now provide background for.  Let $w=w_1w_2\cdots w_n$ be a word, where the $w_i$'s are letters.  For an integer $p\geq 1$, we say that $w$ has \emph{period} $p$ if $w_{i+p}=w_i$ for all $1\leq i\leq n-p$.  Note that $w$ may have many periods; the minimal period of $w$ is called \emph{the} period of $w$.  The \emph{exponent} of $w$ is the length of $w$ divided by the period of $w$.  For a real number $b$, the word $w$ is \emph{$b$-free} if it contains no factor of exponent greater than or equal to $b$, and the word $w$ is \emph{$b^+$-free} if it contains no factor of exponent greater than $b$.  So $2$-free words are exactly the square-free words, and $2^+$-free words are exactly the overlap-free words.

For ease of writing, we unify the notions of $b$-free word and $b^+$-free word by considering $\beta$-free words, where $\beta$ belongs to the set of ``extended real numbers''.  Let $\Rext$ denote the set of extended real numbers, consisting of all real numbers, together with all real numbers with a $+$, where $x^+$ covers $x$, and the inequalities $y\leq x$ and $y<x^+$ are equivalent.  For $\beta\in \Rext$, we say that $w$ is $\beta$-free if no factor of $w$ has exponent greater than or equal to $\beta$.

\begin{definition}
Let $w$ be a word over a fixed alphabet $\Sigma$, and let $\beta\in \Rext$.  We say that $w$ is \emph{extremal $\beta$-free} if $w$ is $\beta$-free, and every extension of $w$ contains a factor of exponent greater than or equal to $\beta$.
\end{definition}

We consider the following problem.

\begin{problem}\label{BinaryProblem}
For which $\beta\in\Rext$ do there exist arbitrarily long extremal $\beta$-free words over $\Sigma_2$?
\end{problem}

On the affirmative side, by Theorem~\ref{main}, we know that there are arbitrarily long extremal $2^+$-free words over $\Sigma_2$.  On the negative side, every binary word of length at least $4$ contains a square, so it follows that for all $\beta\leq 2$, there do not exist arbitrarily long extremal $\beta$-free words over $\Sigma_2$.  We make some further progress on Problem~\ref{BinaryProblem} on the affirmative side by establishing the following theorem.

\begin{theorem}\label{BetaTheorem}
Let $\beta\in\Rext$ satisfy $2^+\leq \beta\leq 8/3$.  Then there are arbitrarily long extremal $\beta$-free words over $\Sigma_2$.
\end{theorem}

We also make the following conjecture.

\begin{conjecture}\label{BinaryConjecture}
There is some number $\alpha\in \Rext$ such that for all $\beta\in \Rext$ satisfying $\beta\geq \alpha$, there are no extremal $\beta$-free words over $\Sigma_2$.
\end{conjecture}

It is possible that Conjecture~\ref{BinaryConjecture} is true with $\alpha=8/3^+$, but we have only very weak computational evidence supporting this.  If one could show that Conjecture~\ref{BinaryConjecture} is true with $\alpha=8/3^+$, then it would completely answer Problem 1.3.

The layout of the remainder of the paper is as follows.  We prove Theorem~\ref{main} in Section~\ref{EvenSection} and Section~\ref{OddSection}.  We consider the even lengths in Section~\ref{EvenSection}, and the odd lengths in Section~\ref{OddSection}.  We prove Theorem~\ref{BetaTheorem} in Section~\ref{BetaSection}.  We conclude with a discussion of some open problems and conjectures over larger alphabets.

\section{Extremal overlap-free words of even length}\label{EvenSection}

In this section, we characterize the even lengths for which there are
extremal overlap-free binary words.  Throughout the remainder of the paper, we let $\mu:\Sigma_2^*\rightarrow \Sigma_2^*$ denote the \emph{Thue-Morse morphism}, defined by $\mu(\tt{0})=\tt{01}$ and $\mu(\tt{1})=\tt{10}$.  The \emph{Thue-Morse word} $\mathbf{t}$ is the unique fixed point of the morphism $\mu$ that begins with $\tt{0}$.  In other words, we have $\mathbf{t}=\mu^\omega(\tt{0})$.  The Thue-Morse word is the prototypical example of a $2$-automatic sequence.  We begin with a lemma that is used frequently in the rest of the paper.

\begin{lemma}\label{MiddleExtensions}
Let $w\in\Sigma_2^*$ be an overlap-free word of length at least $10$, and write $w=w'w''$ with $|w'|,|w''|\geq 5$.  Then for every letter $a\in\Sigma_2$, the extension $w'aw''$ contains an overlap of period at most $3$ (and hence a factor of exponent at least $7/3$).
\end{lemma}

\begin{proof}
It suffices to check the lemma statement for all overlap-free words in $\Sigma_2^*$ of length exactly $10$, which is completed easily by computer.
\end{proof}

\begin{definition}
A word $w\in\Sigma_2^*$ is called \emph{earmarked} if all of the following conditions are satisfied:
\begin{enumerate}
\item $w$ is overlap-free;
\item the length $4$ prefix of $w$ is in $\{\tt{0010},\tt{1101}\}$; and
\item the length $4$ suffix of $w$ is $\tt{0100}$. 
\end{enumerate}
\end{definition}

\begin{lemma}\label{EarmarkedToExtremal}
Let $u$ be an earmarked word of length at least $8$.  Let $w$ be the word obtained from $v=\mu(u)$ by complementing the first and last letters.  Then $w$ is both earmarked and extremal overlap-free.
\end{lemma}

\begin{proof}
Assume that $u$ has prefix $0010$; the case that $u$ has prefix $1101$ is handled similarly.  So we may write $u=\tt{0010}u'\tt{0100}$ for some word $u'\in\Sigma_2^*$.  It follows that 
\[
w=\tt{11011001}\mu(u')\tt{01100100}
\] 
So $w$ has length $4$ prefix $\tt{1101}$ and length $4$ suffix $\tt{0100}$.  

We now show that $w$ is overlap-free.  First of all, since $u$ and $\mu$ are overlap-free, we see that $v$ is overlap-free. Now suppose that $w$ contains the overlap $x$.  Since $v$ is overlap-free, we see that $x$ must be either a prefix or a suffix of $w$.  Assume that $x$ is a prefix of $w$; the case that $x$ is a suffix of $w$ is handled similarly.  Since the word $\tt{11011}$ may only appear as a prefix or a suffix of an overlap-free word, we conclude that the period of $x$ is at most $4$.  But by inspection, there is no such overlap in $w$.

Finally, we show that $w$ is extremal overlap-free.  By Lemma~\ref{MiddleExtensions}, it suffices to check that every extension of $w$ of the form $w'aw''$, where $w=w'w''$, $a\in\Sigma_2$, and either $|w'|\leq 4$ or $|w''|\leq 4$, contains an overlap.  We complete this check by inspection.
\end{proof}

\begin{lemma}\label{TMEarmarked}
Let $n\geq 10$ be an integer satisfying $n\not\equiv 0\pmod{4}$.  Then there is an earmarked word of length $n$.
\end{lemma}

\begin{proof}
We use the automatic theorem-proving software \tt{Walnut}~\cite{Walnut} to show that the Thue-Morse word $\mathbf{t}$ contains a factor $u$ of length $n-4$ such that the word $u\tt{0100}$ is earmarked.  The interested reader can verify our results in \tt{Walnut}; the complete code that we used can be found in Appendix~\ref{Appendix}.  We essentially adapt the predicates used by Clokie, Gabric, and Shallit~\cite[Theorem 1]{ClokieGabricShallit2019}.

\begin{figure}[htb]
\centering
\includegraphics[scale=0.4]{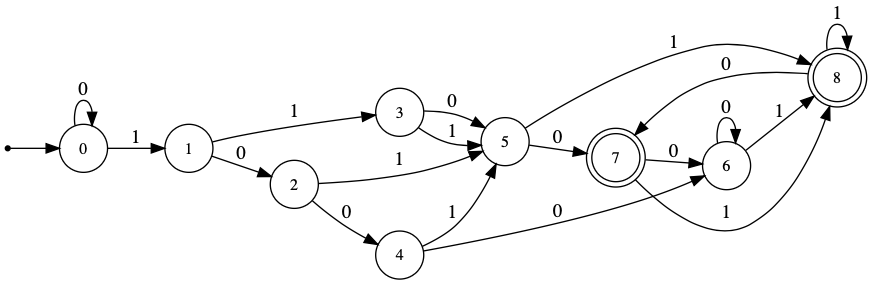}
\caption{The automaton accepting those $(n)_2$ for which the Thue-Morse word contains a factor $v$ of length $n-4$  such that $v\tt{0100}$ is earmarked.}
\label{TMEarmarkedAutomaton}
\end{figure}

First, we create a predicate $\tt{overlap}(i,n,p,s)$ which evaluates to \tt{true} if the word $u\tt{0100}$ contains an overlap of period $p$ with $p\geq 1$ beginning at index $i-s$, where $u=\mathbf{t}[s..s+n-5]$.  We use a straightforward modification of the method described by Clokie, Gabric, and Shallit~\cite[Proof of Theorem~1]{ClokieGabricShallit2019} to do so.  Next, we create a predicate $\tt{earmarked}(n,s)$ which evaluates to \tt{true} if the word $u\tt{0100}$ defined above is earmarked:
\begin{align*}
(&n\geq 8) \land (\mathbf{t}[s..s+3]\in\{\tt{0010},\tt{1101}\})\\ 
&\land (\forall i,p\ ((p\geq 1) \land (i\geq s) \land (i-s+2p <n)) \Rightarrow \lnot(\tt{overlap}(i,n,p,s)))
\end{align*}
Finally, the predicate
\[
\tt{testEarmarked}(n):=\exists s \ \tt{earmarked}(n,s)
\]
evaluates to \emph{true} if there is some length $n-4$ factor $v$ of the Thue-Morse word such that $v\tt{0100}$ is earmarked.  The automaton for $\tt{testEarmarked}(n)$ is shown in Figure~\ref{TMEarmarkedAutomaton}.  By inspection, this automaton accepts all integers $n\geq 10$ such that $n\not\equiv 0\pmod{4}$.
\end{proof}

\begin{lemma}\label{AllEarmarked}
Let $n\geq 10$ be an integer that is not a power of two.  Then there is an earmarked word of length $n$.
\end{lemma}

\begin{proof}
By Lemma~\ref{TMEarmarked}, we may assume that $n\equiv 0\pmod{4}$.  Since $n$ is not a power of two, we may write $2^k<n<2^{k+1}$ for some $k\geq 3$.  We proceed by induction on $k$.  If $k\leq 4$, then $n\in\{12,20,24,28\}$.  It is easily verified by computer that the following words (found by computer search) are earmarked:
\begin{align*}
\text{Length 12: } &\tt{001001100100}\\
\text{Length 20: } &\tt{00100110100101100100}\\
\text{Length 24: } &\tt{110110010110100101100100}\\
\text{Length 28: } &\tt{1101100110100101101001100100}
\end{align*}
So we may assume that $k\geq 5$.  Let $m=n/2$.  Note that $m$ is not a power of two, and that $10<2^{k-1}<m<2^k$.  If $m\not \equiv 0\pmod{4}$, then there is an earmarked word of length $m$ by Lemma~\ref{TMEarmarked}.  If $m\equiv 0\pmod{4}$, then there is an earmarked word of length $m$ by the induction hypothesis.  So either way, there is an earmarked word of length $m$.  By Lemma~\ref{EarmarkedToExtremal}, there is an earmarked word of length $2m=n$.
\end{proof}

\begin{corollary}\label{NotAPowerOfTwo}
Let $n\geq 20$ be an even integer that is not a power of two.  Then there is an extremal overlap-free word of length $n$.
\end{corollary}

\begin{proof}
By Lemma~\ref{AllEarmarked}, there is an earmarked word $u$ of length $m=n/2$.  By Lemma~\ref{EarmarkedToExtremal}, the word $w$ of length $n$ obtained from $\mu(u)$ by complementing the first and last letters is extremal overlap-free.
\end{proof}

\begin{lemma}\label{PowerOfTwo}
For every integer $k\geq 5$, there is an extremal overlap-free word of length $2^k$.
\end{lemma}

\begin{proof}
Let $s=\tt{0}\mu(\tt{0011001})\tt{1}=(\tt{00101101})^2$.
We claim that the word $\mu^\ell(s)$ is extremal overlap-free for every integer $\ell\geq 1$.  Since $s$ has length $16$, the word $\mu^\ell(s)$ has length $2^{\ell+4}$, and hence the theorem statement follows.

Fix $\ell\geq 1$, and let $w=\mu^\ell(s)$.  If $\ell\leq 2$, then we verify that $w$ is extremal overlap-free by computer, so we may assume that $\ell\geq 3$.  First note that $s$ is overlap-free, and hence $w$ is overlap-free.  It remains to show that every extension of $w$ contains an overlap.  Consider an extension $w'aw''$ of $w$, where $w=w'w''$ and $a\in\Sigma_2$.  By Lemma~\ref{MiddleExtensions}, we may assume that $|w'|\leq 4$ or $|w''|\leq 4$.  We consider several cases.

\textbf{Case I:} $|w'|=0$.  Note that $w$ begins with the squares $\mu^\ell(00)$ and $\mu^\ell(s)$.  If $\ell$ is even, then $\mu^\ell(00)$ ends with a $\tt{0}$, and $\mu^\ell(s)$ ends with a $\tt{1}$.  If $\ell$ is odd, then $\mu^\ell(00)$ ends with a $\tt{1}$, and $\mu^\ell(s)$ ends with a $\tt{0}$.  So either way, the extensions $\tt{0}w$ and $\tt{1}w$ both contain an overlap.

\textbf{Case II:} $1\leq |w'|\leq 4$.  Since $\ell\geq 3$, we see that $w$ has prefix $\mu^3(\tt{0})=\tt{01101001}$.
If $|w'|=1$, then the extension $w'\tt{0}w''=\tt{0}w$ contains an overlap by Case I, and the extension $w'\tt{1}w''$ contains the overlap $\tt{111}$.
So we may assume that $2\leq |w'|\leq 4$.  By inspection, the extension $w'aw''$ contains an overlap of period at most $3$.

\textbf{Case III:} $|w''|=0$.  Note that $\mu^\ell(s)$ and $\mu^\ell(\tt{101101})$ are square suffixes of $\mu^\ell(s)$ which begin in $\tt{0}$ and $\tt{1}$, respectively.  So both of the extensions $w\tt{0}$ and $w\tt{1}$ contain an overlap.

\textbf{Case IV:} $1\leq |w''|\leq 4$.  Since $\ell\geq 3$, we see that $w$ has suffix $\mu^3(\tt{0})=\tt{01101001}$ if $\ell$ is even, and suffix $\mu^3(\tt{1})=\tt{10010110}$ if $\ell$ is odd.  Either way, the remainder of the proof is similar to that of Case II.
\end{proof}

\begin{proposition}\label{EvenLengths}
Let $n$ be a nonnegative even number.  Then there is an extremal overlap-free word of length $n$ over the alphabet $\Sigma_2$ if and only if $n\in\mathcal{N}$.
\end{proposition}

\begin{proof}
If $n\in\{0,2,4,6,8,14,16,18\}$, then an exhaustive backtracking search shows that no extremal overlap-free word of length $n$ exists over $\Sigma_2$.  The words $\tt{0010011011}$ and $\tt{001001100100}$, of lengths $10$ and $12$, respectively, are extremal overlap-free.
So suppose that $n\geq 20$. If $n$ is a power of two, then there is an extremal overlap-free word of length $n$ by Corollary~\ref{PowerOfTwo}.  If $n$ is not a power of two, then there is an extremal overlap-free word of length $n$ by Lemma~\ref{NotAPowerOfTwo}.
\end{proof}

\section{Extremal overlap-free words of odd length}\label{OddSection}

In this section, we characterize the odd lengths for which there are
extremal overlap-free binary words.  We need two classical results
from the theory of overlap-free binary words.  The first is the
so-called \emph{factorization theorem} of Restivo and Salemi
\cite{RestivoSalemi1985} (see also
\cite[Proposition~1.7.5(a)]{AlloucheShallit}).

\begin{theorem}
Let $x \in \lbrace \mathtt{0,1} \rbrace^*$ be overlap-free.  Then there exist
$u, v \in \lbrace \varepsilon, \mathtt{0, 1, 00, 11} \rbrace$ and an
overlap-free word $y$ such that $x = u \mu(y) v$.  Furthermore, this
factorization is unique if $|x| \geq 7$.
\label{res-sal}
\end{theorem}

Words $u$ and $v$ are \emph{conjugates} if there exist words $x$ and
$y$ such that $u=xy$ and $v=yx$, i.e., if they are cyclic shifts of
one another.  Let $w\in\Sigma^*$.  The \emph{circular word} formed
from $w$ is the set of all conjugates of $w$.  Thue~\cite[Proposition~2.13 (Satz~13)]{Berstel1995} characterized the
circular overlap-free binary words, which also yields a
characterization of the overlap-free binary squares (see also the work
of Shelton and Soni \cite{SheltonSoni1985}).

Define \[A = \{\mathtt{00,11,010010,101101}\}\] and
\[\mathcal{A} = \bigcup_{k \geq 0}\mu^k(A).\]

\begin{theorem}\label{ovlpf-sqs}
The overlap-free binary squares are the conjugates of the
words in $\mathcal{A}$.
\end{theorem}

\begin{remark}\label{sq_types}
From Theorems~\ref{res-sal} and \ref{ovlpf-sqs}, we deduce that if $vv$
is an overlap-free binary square of length greater than $6$, then $vv$
can be written in exactly one of the following two forms: $vv =
\mu(zz)$ or $vv = \overline{a}\mu(z)a$ for some $a \in \{\mathtt{0,1}\}$ and
some $z \in \{\mathtt{0,1}\}^*$.
\end{remark}

\begin{proposition}\label{OddLengthExistence}
  Let $u$ be an extremal overlap-free binary word of odd length.  Then
  either $|u| = 2^{k}+1$ or $|u| = 3\cdot 2^k+1$ for some $k$.
\end{proposition}

\begin{proof}
  By Theorem~\ref{res-sal}, we can, without loss of generality,
  consider two possible forms for $u$: either $u = \mu(y)a$ or
  $u = bb\mu(y)a$ for some $a,b \in \{\mathtt{0,1}\}$.  If $u$ is extremal
  overlap-free, then both $ua$ and $u\overline{a}$ end in overlaps.
  Consequently, the word $u$ ends in at least two distinct squares.
  Let $vv$ be the longest square suffix of $u$.

  Suppose first that $|vv|>6$.  By Remark~\ref{sq_types}, we see that
  $vv = \overline{a}\mu(z)a$ for some word $z$.  If $vv$ is a proper
  factor of $\mu(y)a$, then $vv$ is preceded by $a$ in $u$; however,
  since $v$ ends with $a$, the word $avv$ is an overlap in $u$, which
  is a contradiction.  We conclude that
  $u = \overline{a}\overline{a}\mu(z)a=\overline{a}vv$, and hence that
  either $|u| = 2^{k}+1$ or $|u| = 3\cdot 2^k+1$ for some $k$, as
  required.

  Now consider the case $|vv| \leq 6$.  Since $u$ ends in two distinct
  squares, these squares are both conjugates of words in
  $A \cup \{\mathtt{0101}\}$, and, since one must be a suffix of the
  other, we observe that the only possibilities for these two squares
  are $aa$ and $\overline{a}aa\overline{a}aa$.  However,
  $\overline{a}aa\overline{a}aa$ is not a suffix of a word of either
  the form $\mu(y)a$ or the form $bb\mu(y)a$.  This contradiction
  completes the proof.
\end{proof}

The proof of Lemma~\ref{OddLengthExistence} tells us that any extremal overlap-free word of odd length can be obtained from an overlap-free square by adding a single letter at either the beginning or the end.  This led us to the constructions of overlap-free words of odd length given in the next two lemmas.

\begin{lemma}\label{PowerOfTwoPlusOne}
For every integer $k\geq 5$, there is an extremal overlap-free word of length $2^k+1$.
\end{lemma}

\begin{proof}
Fix $k\geq 5$.  Let $u=(\tt{011})^{-1}\mu^{k-1}(\tt{00})\tt{011}$.  Note that $u$ is a conjugate of $\mu^{k-1}(\tt{00})$.  In particular, we have that $u$ is a square of length $2^k$, and by Theorem~\ref{ovlpf-sqs}, we see that $u$ is overlap-free.  We claim that the word $v=\tt{0}u$
is extremal overlap-free.  We first show that $v$ is overlap-free.  Since $u$ is overlap-free, it suffices to show that no prefix of $v$ is an overlap.  Since $v$ has prefix $\tt{00100}$, which never appears again in $v$, it suffices to check that $u$ does not begin with an overlap of period at most $4$, which is easily done by inspection.

It remains to show that every extension of $v$ contains an overlap.  Consider an extension $v'av''$ of $v$, where $v=v'v''$ and $a\in\Sigma_2$.  By Lemma~\ref{MiddleExtensions}, we may assume that $|v'|\leq 4$ or $|v''|\leq 4$.  First suppose that $|v'|\leq 4$.  Note that $v$ has prefix $\tt{0}(\tt{011})^{-1}\mu^4(\tt{0})=\tt{00100110010110}$.  By inspection, the extension $v'av''$ contains an overlap of period at most $4$.  Now suppose that $|v''|\leq 4$.  Since $u$ is a square with first letter $\tt{0}$, and $u$ ends in the square $\tt{11}$, the extension $va$ contains an overlap.  Thus we may assume that $1\leq |v''|\leq 4$.  If $k$ is even, then $v$ has suffix $\mu^4(\tt{0})\tt{011}$, and by inspection, the extension $v'av''$ contains an overlap of period at most $6$.  If $k$ is odd, then $v$ has suffix $\mu^4(\tt{1})\tt{011}$, and by inspection, the extension $v'av''$ contains an overlap of period at most $6$.
\end{proof}

\begin{lemma}\label{ThreeTimesPowerOfTwoPlusOne}
For every integer $k\geq 3$, there is an extremal overlap-free word of length $3\cdot 2^k+1$.
\end{lemma}

\begin{proof}
Fix $k\geq 3$.  Let $u=(\tt{011})^{-1}\mu^{k-1}(\tt{010010})\tt{011}$.  Note that $u$ is a conjugate of $\mu^{k-1}(\tt{010010})$.  In particular, we have that $u$ is a square of length $3\cdot 2^k$, and by Theorem~\ref{ovlpf-sqs}, we see that $u$ is overlap-free.  We claim that the word $v=\tt{0}u$
is extremal overlap-free.  The remainder of the proof is strictly analogous to the proof of Lemma~\ref{PowerOfTwoPlusOne}.
\end{proof}

We now prove the analogue of Proposition~\ref{EvenLengths} for odd $n$.

\begin{proposition}\label{OddLengths}
Let $n$ be a nonnegative odd number.  Then there is an extremal overlap-free word of length $n$ over the alphabet $\Sigma_2$ if and only if $n\in\mathcal{N}$.
\end{proposition}

\begin{proof}
$(\Leftarrow)$ Let $n\in\mathcal{N}$.  Since $n$ is odd, we must have either $n=2^k+1$ for some $k\geq 5$, or $n=3\cdot 2^k+1$ for some $k\geq 3$.  In the former case, there is an extremal overlap-free word of length $n$ by Lemma~\ref{PowerOfTwoPlusOne}, and in the latter case, there is an extremal overlap-free word of length $n$ by Lemma~\ref{ThreeTimesPowerOfTwoPlusOne}.

$(\Rightarrow)$ Suppose that there is an extremal overlap-free word of length $n$ over the alphabet $\{0,1\}$.  By Proposition~\ref{OddLengthExistence}, we must have $n=2^k+1$ or $n=3\cdot 2^k+1$ for some $k$.  By exhaustive computer search, there is no extremal overlap-free word of length $2^k+1$ for $k\leq 4$, and no extremal overlap-free word of length $3\cdot 2^k+1$ for $k\leq 2$.  Thus, we conclude that $n\in\mathcal{N}$.
\end{proof}

Together, Proposition~\ref{EvenLengths} and Proposition~\ref{OddLengths} give Theorem~\ref{main}.

\section{\texorpdfstring{Extremal \boldmath{$\beta$}-free binary words}{Extremal beta-free binary words}}\label{BetaSection}

This section is devoted to the proof of Theorem~\ref{BetaTheorem}.  Another definition facilitates our proof method.

\begin{definition}
Let $w$ be a word over a fixed alphabet $\Sigma$, and let $\alpha,\beta\in \Rext$ satisfy $1<\alpha\leq \beta$.  We say that $w$ is \emph{$(\alpha,\beta)$-extremal} if $w$ is $\alpha$-free, and every extension of $w$ contains a factor of exponent greater than or equal to $\beta$.
\end{definition}

If $w$ is $(\alpha,\beta)$-extremal, then for any $\gamma\in \Rext$ such that $\alpha\leq \gamma\leq \beta$, the word $w$ is extremal $\gamma$-free.  Thus, the following result immediately implies Theorem~\ref{BetaTheorem}.

\begin{proposition}\label{Levels} All of the following hold.
\begin{enumerate}[label=\textnormal{(\alph*)}]
\item\label{Level1} There are arbitrarily long $(2^+,7/3)$-extremal binary words.
\item\label{Level2} There are arbitrarily long $(7/3^+,17/7)$-extremal binary words.
\item\label{Level3} There are arbitrarily long $(17/7^+,5/2)$-extremal binary words.
\item\label{Level4} There are arbitrarily long $(5/2^+,18/7)$-extremal binary words.
\item\label{Level5} There are arbitrarily long $(18/7^+,8/3)$-extremal binary words.
\end{enumerate}
\end{proposition}

We prove the first part of Proposition~\ref{Levels} now.

\begin{proof}[Proof of Proposition~\ref{Levels}\ref{Level1}]
Let $u$ be a factor of the Thue-Morse word of the form $\tt{011}v\tt{110}$, where $v$ is a nonempty word.  Note that there are arbitrarily long words of this form.  We claim that the word $x=\tt{00}\mu^2(\tt{11}v\tt{11})\tt{00}$ is $(2^+,7/3)$-extremal.

First we show that $x$ is $2^+$-free (or in other words, overlap-free).  Since $u$ is a factor of the Thue-Morse word, we have that $u$, and hence $\mu^2(u)$, are overlap-free.  Since the word $\mu^2(u)$ contains the word $\tt{0}\mu^2(\tt{11}v\tt{11})\tt{0}$ as a factor, any overlap contained in $x$ must be either a prefix or a suffix of $x$.  Suppose without loss of generality that $x$ contains an overlap $z$ as a prefix.  Since the factor $\tt{00100}$ does not appear in the Thue-Morse word, this factor appears only as a prefix and a suffix of $x$.  So $z$ must have period at most $4$.  But this is impossible by inspection.

It remains to show that every extension of $x$ contains a factor of exponent at least $7/3$.  Consider an extension $x'ax''$ of $x$, where $x=x'x''$ and $a\in\Sigma_2$.  By Lemma~\ref{MiddleExtensions}, we may assume that $|x'|\leq 4$ or $|x''|\leq 4$.  First suppose that $|x'|\leq 4$.  Note that $x$ has prefix $\tt{00}\mu^2(\tt{11})=\tt{0010011001}$.  By inspection, the extension $x'ax''$ contains a factor of exponent at least $7/3$.  The case that $|x''|\leq 4$ is handled by a symmetric argument.
\end{proof}

One of the main tools that we use to prove Proposition~\ref{Levels} parts~\ref{Level2}-\ref{Level5} is the following extension of a lemma due to Ochem~\cite[Lemma~2.1]{Ochem2006}.  A morphism $f:\Sigma^*\rightarrow \Delta^*$ is called \emph{$q$-uniform} if $|f(a)|=q$ for all $a\in\Sigma$, and is called \emph{synchronizing} if for any $a,b,c\in\Sigma$ and $u,v\in \Delta^*$, if $f(ab)=uf(c)v$, then either $u=\varepsilon$ and $a=c$, or $v=\varepsilon$ and $b=c$.

\begin{lemma}\label{OchemExtension}
Let $a,b\in\mathbb{R}$ satisfy $1<a<b$.  Let $\alpha\in\{a,a^+\}$ and $\beta\in\{b,b^+\}$.  Let $h\colon \Sigma^*\rightarrow \Delta^*$ be a synchronizing $q$-uniform morphism.  If $h(w)$ is $\beta$-free for every $\alpha$-free word $w$ such that
\[
|w|\leq \max\left\{\frac{2b}{b-a},\frac{2(q-1)(2b-1)}{q(b-1)}\right\},
\]
then $h(z)$ is $\beta$-free for every $\alpha$-free word $z\in\Sigma^*$.
\end{lemma}

\begin{proof}
Suppose that there is an $\alpha$-free word $w$ such that the word $W=h(w)$ contains a factor of exponent greater than or equal to $\beta$, and assume without loss of generality that $w$ is a shortest word satisfying this property.  We will show that $|w|\leq \max\left\{\frac{2b}{b-a},\frac{2(q-1)(2b-1)}{q(b-1)}\right\}$, which gives the theorem statement.

Let $X$ be a factor of $W$ of exponent greater than or equal to $\beta$.  Let $P$ be the period of $X$, and write $X=UV$, where $|U|=P$.  Since $X$ has period $P$, we can also write $X=VU'$ for some word $U'\in \Delta^*$.  Let $R=|V|$.  Then we have $\frac{P+R}{P}\geq b$, or equivalently $P\leq\frac{R}{b-1}$.

First suppose that $R\leq 2q-2=2(q-1)$.  Then we have
\[
|X|=P+R\leq\tfrac{R}{b-1}+R=\frac{Rb}{b-1}\leq \frac{2(q-1)b}{b-1}.
\]
By the minimality of $w$, we must have $|w|\leq \frac{|X|-2}{q}+2$.  Putting this together with the above bound on $|X|$, we find $|w|\leq \frac{2(q-1)(2b-1)}{q(b-1)}$.

Now suppose that $R\geq 2q-1$.  Write $V=V_1h(v)V_2$ for some word $v\in \Sigma^*$, where the word $V_1$ is a proper suffix of a block of $h$, and the word $V_2$ is a proper prefix of a block of $h$.  Let $r=|v|$.  Since $R\geq 2q-1$, we must have $r\geq 1$.  Further, since $|V_1|,|V_2|<q$, we have $R<qr+2q$.  Similarly, write $X=X_1h(x)X_2$ for some word $x\in\Sigma^*$, where the word $X_1$ is a proper suffix of a block of $h$, and the word $X_2$ is a proper prefix of a block of $h$.  Since $X=UV=VU'$, and since $h$ is synchronizing, it must be the case that $X_1=V_1$ and $X_2=V_2$.  It follows that we may write $x=uv=vu'$ for some words $u,u'\in\Sigma^*$, i.e., the word $x$ has period $|u|$.  Let $p=|u|$.  Note that $h(u)=V_1^{-1}UV_1$, so $|h(u)|=|U|=P$, and hence $qp=P$.

Since $w$ is $\alpha$-free, we must have $\frac{p+r}{p}\leq a$, or equivalently $r\leq (a-1) p$.  Now 
\[
qp=P\leq\frac{R}{b-1}<\frac{qr+2q}{b-1}\leq q\cdot\frac{r+2}{b-1}\leq q\cdot \frac{(a-1)p+2}{b-1},
\]
from which we conclude that $p<\frac{(a-1)p+2}{b-1}$, or equivalently, that $p<\frac{2}{b-a}$.  Finally, by the minimality of $w$, we must have
\[
|w|\leq 2+p+r\leq 2+ap<2+\frac{2a}{b-a}=\frac{2b}{b-a}.
\]
We conclude in either case that $|w|\leq \max\left\{\frac{2b}{b-a},\frac{2(q-1)(2b-1)}{q(b-1)}\right\}$, as desired.
\end{proof}

We are now ready to prove the remaining parts of Proposition~\ref{Levels}.  We use the following terminology in the proof.  Let $w$ be a word over a fixed alphabet $\Sigma$.   A \emph{left extension} of $w$ is a word of the form $aw$, where $a\in\Sigma$.  A \emph{right extension} of $w$ is a word of the form $wa$, where $a\in\Sigma$.  An \emph{internal extension} of $w$ is a word of the form $w'aw''$, where $|w'|,|w''|\geq 1$, we have $a\in \Sigma$, and $w'w''=w$.

\begin{proof}[Proof of Proposition~\ref{Levels}\ref{Level2}]
Let $u\in\Sigma_3^*$ be a square-free word of length at least $3$, and write $u=avb$, where $a,b\in\Sigma_3$.  Define $f:\Sigma_3^*\rightarrow\Sigma_2^*$ by
\begin{align*}
f(\tt{0})&=\tt{001011001101100100110100110110010011}\\
f(\tt{1})&=\tt{001011001101100100110110010110010011}\\
f(\tt{2})&=\tt{001011001101100101100100110110010011}.
\end{align*} 
Let $r=\tt{1100110010011}$ and $s=\tt{001001}$.  We claim that the word $w=rf(v)s$ is $(7/3^+,17/7)$-extremal.

First of all, we verify the following statements by computer for every letter $c\in\Sigma_3$: 
\begin{itemize}
\item Every internal extension of the word $f(c)$ contains a factor of exponent at least $17/7$.
\item Every left extension and every internal extension of the word $rf(c)$ contains a factor of exponent at least $17/7$.
\item Every right extension and every internal extension of the word $f(c)s$ contains a factor of exponent at least $17/7$.
\end{itemize}
It now follows easily that every extension of the word $w=rf(v)s$ contains a factor of exponent at least $17/7$.  The only extensions of $w$ not checked above are those obtained by inserting a letter between two blocks  of $f$.  Since every block of $f$ begins in $\tt{00}$ and ends in $\tt{11}$, every such extension contains a cube.

It remains to show that $w$ is $7/3^+$-free.  We first show that $f(u)$ is $7/3^+$-free.  Note that $f$ is $36$-uniform, and we verify by computer that $f$ is synchronizing.  Thus, by Lemma~\ref{OchemExtension}, it suffices to check that $f(x)$ is $7/3^+$-free for every square-free word $x\in\Sigma_3^*$ such that $|x|\leq 14$, which we verify by computer.  Note that every block of $f(u)$ has prefix $s'=\tt{0010}$ and suffix $r'=\tt{0110010011}$.  So $f(u)$ contains the word $w'=r'f(v)s'$, and hence $w'$ is $7/3^+$-free.  Note that $s=s'\tt{01}$ and $r=\tt{110}r'$, so 
\[
w=rf(v)s=\tt{110}r'f(v)s'\tt{01}.
\]

Suppose that $w$ contains a factor $z$ of exponent greater than $7/3$.  Then $z$ begins at one of the first three letters of $w$, or ends at one of the last two letters of $w$.  Suppose first that $z$ begins at one of the first three letters of $w$.  We claim that the factor $t_r=\tt{00110010011}$, which occurs starting at the third letter of $w$, occurs only once in $w$.  To establish this claim, we verify the following by computer:
\begin{itemize}
\item For every $c\in\Sigma_3$, the word $t_r$ occurs exactly once in the word $rf(c)$, and does not occur in the word $f(c)s$.
\item For every square-free word $y\in\Sigma_3^*$ of length $2$, the word $t_r$ does not occur in $f(y)$.    
\end{itemize}
So we see that the period of $z$ is at most $13$.  However, this possibility is ruled out by computer check.  So we may assume that $z$ ends at one of the last two letters of $w$.  By a computer check similar to the one used for $t_r$, we verify that the factor $t_s=\tt{001001100100}$, which occurs ending at the second last letter of $w$, occurs only once in $w$.  So again, we see that the period of $z$ is at most $13$.  This possibility is ruled out by computer check.  
\end{proof}

We omit the details of the proofs of Proposition~\ref{Levels}\ref{Level3}--\ref{Level5}, as they are similar to the proof of Proposition~\ref{Levels}\ref{Level2}.

\begin{proof}[Proof of Proposition~\ref{Levels}\ref{Level3}]
Let $u\in\Sigma_3^*$ be a square-free word of length at least $3$, and write $u=avb$, where $a,b\in\Sigma_3$.  Define $f:\Sigma_3^*\rightarrow\Sigma_2^*$ by
\begin{align*}
f(\tt{0})&=\tt{001001100101100100110010110100110010110011011}\\
f(\tt{1})&=\tt{001001100101100110100101100110110010110011011}\\
f(\tt{2})&=\tt{001001100101100110100110110011010010110011011}.
\end{align*} 
Let $r=\tt{00110110011011}$ and $s=\tt{00100110010011}$.  By a method similar to the one used in the proof of Proposition~\ref{Levels}\ref{Level2}, one can show that the word $w=rf(v)s$ is $(17/7^+,5/2)$-extremal.
\end{proof}

\begin{proof}[Proof of Proposition~\ref{Levels}\ref{Level4}]
Let $u\in\Sigma_3^*$ be a square-free word of length at least $3$, and write $u=avb$, where $a,b\in\Sigma_3$.  Define $f:\Sigma_3^*\rightarrow\Sigma_2^*$ by
\begin{align*}
f(\tt{0})&=\tt{0011011001001100101100110110010011}\\
f(\tt{1})&=\tt{0011011001001101001101100110010011}\\
f(\tt{2})&=\tt{0011011001001101100110100110010011}.
\end{align*} 
Let $r=\tt{00110110011011001010011}$ and $s=\tt{00110101100100110010011}$.  By a method similar to the one used in the proof of Proposition~\ref{Levels}\ref{Level2}, one can show that the word $w=rf(v)s$ is $(5/2^+,18/7)$-extremal.
\end{proof}

\begin{proof}[Proof of Proposition~\ref{Levels}\ref{Level5}]
Let $u\in\Sigma_3^*$ be a square-free word of length at least $3$, and write $u=avb$, where $a,b\in\Sigma_3$.  Define $f:\Sigma_3^*\rightarrow\Sigma_2^*$ by%
\begin{align*}
f(\tt{0})&=\tt{0011011001001100101100110110010011}\\
f(\tt{1})&=\tt{0011011001001101100110100110010011}\\
f(\tt{2})&=\tt{0011011001101001100100110110010011}.
\end{align*} 
Let $r=\tt{01101100110110011001010011}$ and $s=\tt{00110101100110010011001001}$.  By a method similar to the one used in the proof of Proposition~\ref{Levels}\ref{Level2}, one can show that the word $w=rf(v)s$ is $(18/7^+,8/3)$-extremal.
\end{proof}

\section{Conclusion}

In this paper, we have focused on extremal $\beta$-free words over the binary alphabet $\Sigma_2$.  First, we characterized the lengths of extremal $2^+$-free (i.e., overlap-free) words over $\Sigma_2$.  We then made some significant progress on Problem~\ref{BinaryProblem} by establishing that there are arbitrarily long extremal $\beta$-free words over $\Sigma_2$ for every $\beta\in \Rext$ such that $2^+\leq \beta\leq 8/3$.  Problem~\ref{BinaryProblem} remains open for $\beta\geq 8/3^+$.

We close with a discussion of some related problems over larger alphabets.  First of all, we have the following general problem which subsumes Problem~\ref{BinaryProblem}.

\begin{problem}\label{BigProblem}
Let $n\geq 2$ be an integer.  For which $\beta\in\Rext$ do there exist arbitrarily long extremal $\beta$-free words over $\Sigma_n$?
\end{problem}

For every integer $n\geq 2$, let $B_n$ denote the set of all $\beta\in \Rext$ such that there exist arbitrarily long extremal $\beta$-free words over $\Sigma_n$.  While it seems plausible that $B_n$ is an interval for every $n$, it is not immediately obvious to us that this is the case.

We note that Dejean's theorem gives us a partial answer to Problem~\ref{BigProblem}.  The \emph{repetition threshold} for $n$ letters, denoted $\RT(n)$, is defined by
\[
\RT(n)=\inf\{b\in\mathbb{R}\colon\ \text{there are arbitrarily long $b$-free words over $\Sigma_n$}\}.
\]
Dejean's theorem, originally conjectured by Dejean~\cite{Dejean1972}, and confirmed through the work of many authors~\cite{Dejean1972, CurrieRampersad2009,CurrieRampersad2009Again,CurrieRampersad2011,Rao2011, Carpi2007, MoulinOllagnier1992,Pansiot1984}, states that
\[
\RT(n)=\begin{cases}
2, &\text{if $n=2$;}\\
7/4, &\text{if $n=3$;}\\
7/5, &\text{if $n=4$;}\\
n/(n-1), &\text{if $n\geq 5$}.
\end{cases}
\]
In fact, for every $n\geq 2$, it is known that there are only finitely many $\RT(n)$-free words over $n$ letters, but infinitely many $\RT(n)^+$-free words over $n$ letters.  Thus, if there are arbitrarily long extremal $\beta$-free words over $\Sigma_n$, then $\beta> \RT(n)$.

\begin{conjecture}\label{RTconjecture}
For every $n\geq 2$, there are arbitrarily long extremal $\RT(n)^+$-free words over $\Sigma_n$.
\end{conjecture}

We define the \emph{extremal repetition threshold} over $n$ letters, denoted $\ERT(n)$, by
\[
\ERT(n)=\sup\left\{b\in\mathbb{R}\colon\ \text{there are arbitrarily long extremal $b^+$-free words over $\Sigma_n$}\right\}.
\]
By Theorem~\ref{BetaTheorem}, we know that $\ERT(2)\geq 8/3$.  From the work of Grytczuk et al.~\cite{Grytczuk2019}, we know that $\ERT(3)\geq 2$.  It may be the case that $\ERT(2)=8/3$ and $\ERT(3)=2$, but we have only weak computational evidence supporting this. 

If Conjecture~\ref{RTconjecture} is true, then $\ERT(n)\geq \RT(n)$ for every $n\geq 2$.  We conjecture further that $\ERT(n)$ is finite for every $n\geq 2$.  In fact, we make the following stronger conjecture, which subsumes Conjecture~\ref{BinaryConjecture}.

\begin{conjecture}\label{AlphaConjecture}
Let $n\geq 2$ be an integer.  Then there is some number $\alpha_n\in \Rext$ such that for all $\beta\in \Rext$ satisfying $\beta\geq \alpha_n$, there are no extremal $\beta$-free words over $\Sigma_n$.
\end{conjecture}

We close with the following problem, which appears to be quite difficult.

\begin{problem}
For every $n\geq 2$, find $\ERT(n)$ and the smallest number $\alpha_n$ for which Conjecture~\ref{AlphaConjecture} holds (if the conjecture is true).  It is possible that we have $\alpha_n=\ERT(n)^+$ for every $n$.
\end{problem}

\section*{Acknowledgements}

The authors wish to thank Trevor Clokie for helpful discussions.

\begin{appendices}
\section{}\label{Appendix}

The free software \tt{Walnut} used in the proof of Lemma~\ref{TMEarmarked} is available at
\url{https://github.com/hamousavi/Walnut},
and a manual for its use is~\cite{Walnut}.  The complete \tt{Walnut} code used in the proof of Lemma~\ref{TMEarmarked} is given below.

{\footnotesize
\begin{verbatim}
def overlap "(n >=8) & (1 <= p) & (s <= i) & (i+2*p < s+n)                     &
    (Aj ((j>=i)&(j<i+p+1)  &               (j+p < s+n-4) ) =>  T[j]  = T[j+p]) &
    (Aj ((j>=i)&(j<i+p+1)  &               (j+p = s+n-4) ) =>  T[j]  =  @0)    &
    (Aj ((j>=i)&(j<i+p+1)  & (j < s+n-4) & (j+p = s+n-3) ) =>  T[j]  =  @1)    &
    (Aj ((j>=i)&(j<i+p+1)  & (j < s+n-4) & (j+p = s+n-2) ) =>  T[j]  =  @0)    &
    (Aj ((j>=i)&(j<i+p+1)  & (j < s+n-4) & (j+p = s+n-1) ) =>  T[j]  =  @0)    &
    (Aj ~((j>=i)&(j<i+p+1) & (j = s+n-4) & (j+p = s+n-3)))                     &
    (Aj ~((j>=i)&(j<i+p+1) & (j = s+n-3) & (j+p = s+n-2)))                     &
    (Aj ~((j>=i)&(j<i+p+1) & (j = s+n-3) & (j+p = s+n-1)))":
    
def earmarked "(n>=8) & 
    (((T[s]= @0) & (T[s+1]= @0) & (T[s+2]= @1) & (T[s+3]= @0)) |
    ((T[s]= @1) & (T[s+1]= @1) & (T[s+2]= @0) & (T[s+3]= @1)))    &
    (Ai,p ((1 <= p) & (s <= i) & (i+2*p < s+n)) => ~($overlap(i,n,p,s)))":

def testEarmarked "Es $earmarked(n,s)":
\end{verbatim}}

\end{appendices}


\begin{thebibliography}{1}
\bibitem{AlloucheShallit}
  J.-P. Allouche and J. Shallit, \emph{Automatic Sequences}, Cambridge,
  2003.
  
	\bibitem{Berstel1995}
J. Berstel, Axel Thue's papers on repetitions in words: A translation,
  \emph{Publications du LaCIM (Universit{\'{e}} du Qu{\'{e}}bec {\`{a}}
  Montr{\'{e}}al)}, vol.~20, 1995.
    
  	\bibitem{BerstelPerrin2007}
  	J. Berstel and D. Perrin, The origins of combinatorics on words, \emph{European J. Combin.} \textbf{28} (2007), 996-1022.
    
    \bibitem{ClokieGabricShallit2019} T. Clokie, D. Gabric, and J. Shallit, Circularly squarefree words and unbordered conjugates: A new approach, in \emph{Proc. 8th Intl. Conf. Words 2019 (WORDS 2019)}, Lect. Notes in Comput. Sci., Vol. 11682, Springer, 2019, pp. 133--144. 
    
\bibitem{Carpi2007}
A.~Carpi, {On Dejean's conjecture over large alphabets}, {\em Theoret. Comput.
  Sci.} {\bf 385} (2007), 137--151.

\bibitem{CurrieRampersad2009Again}
J.~D. Currie and N.~Rampersad, {Dejean's conjecture holds for $n\geq 27$}, {\em
  RAIRO - Theor. Inform. Appl.} {\bf 43} (2009), 775--778.

\bibitem{CurrieRampersad2009}
J.~D. Currie and N.~Rampersad, {Dejean's conjecture holds for $n\geq 30$}, {\em
  Theoret. Comput. Sci.} {\bf 410} (2009), 2885--2888.

\bibitem{CurrieRampersad2011}
J.~D. Currie and N.~Rampersad, {A proof of Dejean's conjecture}, {\em Math.
  Comp.} {\bf 80} (2011), 1063--1070.

\bibitem{Dejean1972}
F.~Dejean, {Sur un th{\'{e}}or{\`{e}}me de Thue}, {\em J. Combin. Theory Ser.
  A} {\bf 13} (1972), 90--99.
    
	\bibitem{Grytczuk2019} J. Grytczuk, H. Kordulewski, and A. Niewiadomski,  Extremal square-free words, \emph{Electron. J. Combin.} \textbf{27} (2020), \#P1.48.
	    
    
    \bibitem{LothaireAlgebraic}
M. Lothaire, \emph{{Algebraic Combinatorics on Words}}, Cambridge University
  Press, 2002.

\bibitem{MohammadNooriCurrie2007}
M.~Mohammad-Noori and J.~D. Currie, {Dejean's conjecture and Sturmian words},
  {\em European J. Combin.} {\bf 28} (2007), 876--890.
  
  \bibitem{MolRampersad2020}
L. Mol and N. Rampersad, Lengths of extremal square-free ternary words, preprint, 2020.  Available at \url{https://arxiv.org/abs/2001.11763}.

\bibitem{MoulinOllagnier1992}
J.~Moulin-Ollagnier, {Proof of Dejean's conjecture for alphabets with 5, 6, 7,
  8, 9, 10, and 11 letters}, {\em Theoret. Comput. Sci.} {\bf 95} (1992),
  187--205.

\bibitem{Walnut}
H. Mousavi, Automatic theorem proving in Walnut, preprint, 2016.  Available at \url{https://arxiv.org/abs/1603.06017}.
    

\bibitem{Ochem2006}
P. Ochem, A generator of morphisms for infinite words, \emph{RAIRO -- Theoret. Inform. Appl.} \textbf{40} (2006), 427--441.

\bibitem{Pansiot1984}
J.~J. Pansiot, {A propos d'une conjecture de F. Dejean sur les
  r{\'{e}}p{\'{e}}titions dans les mots}, {\em Discrete Appl. Math.} {\bf 7}
  (1984), 297--311.

\bibitem{Rao2011}
M.~Rao, {Last cases of Dejean's conjecture}, {\em Theoret. Comput. Sci.} {\bf
  412} (2011), 3010--3018.

  \bibitem{RestivoSalemi1985}
    A. Restivo and S. Salemi, Overlap-free words on two symbols, in \emph{ Automata on infinite words}, Lect. Notes in Comput. Sci., Vol. 192, Springer, 1985, 198--206.
    
  \bibitem{SheltonSoni1985}
    R. O. Shelton and R. P. Soni, Chains and fixing blocks in irreducible binary sequences, \emph{Discrete Math.} \textbf{54} (1985), 93--99.
    
    \bibitem{Thue1906}
A.~Thue.
\newblock {\"Uber} unendliche {Zeichenreihen}.
\newblock {\em Norske vid. Selsk. Skr. Mat. Nat. Kl.} {\bf 7} (1906), 1--22.
\newblock Reprinted in {\it Selected Mathematical Papers of Axel Thue},
T. Nagell, editor, Universitetsforlaget, Oslo, 1977, pp.~139--158.

\bibitem{Thue1912}
A.~Thue.
\newblock {\"Uber} die gegenseitige {Lage} gleicher {Teile} gewisser
   {Zeichenreihen}.
\newblock {\em Norske vid. Selsk. Skr. Mat. Nat. Kl.} {\bf 1} (1912), 1--67.
\newblock Reprinted in {\it Selected Mathematical Papers of Axel Thue},
T. Nagell, editor, Universitetsforlaget, Oslo, 1977, pp.~413--478.
\end{thebibliography}
\end{document}